\documentclass[11pt]{amsart} \textheight 22cm

\usepackage{hhline}
\usepackage{amssymb}
\usepackage[dvips, dvipsnames, usenames]{color}
\usepackage{eucal}

\newcommand{\fd}{finite-dimensional}

\newcommand{\Ind}{\operatorname{Ind}}
\newcommand{\Imm}{\operatorname{Im}}

\newcommand{\supp}{\operatorname{supp}}

\newcommand{\oct}{\mathfrak O}
\newcommand\sigmae{\sigma_{e}}
\newcommand\sigmao{\sigma_{o}}
\newcommand\so{\sigmao}

\newcommand{\fp}{{\mathbb F}_p}

\newcommand\toba{{\mathfrak B }}

\newcommand{\trid}{\triangleright}

\newcommand{\R}{{\mathcal R}}

\newcommand{\ku}{\mathbb C}

\newcommand{\Z}{{\mathbb Z}}
\newcommand{\N}{{\mathbb N}}

\newcommand{\Q}{{\mathsf Q}}
\newcommand{\F}{{\mathbb F}}

\newcommand{\D}{{\mathcal D}}

\newcommand{\q}{{\mathbf q}}

\newcommand{\Ee}{{\mathcal E}}

\newcommand{\GL}{\mathbf{GL}}

\newcommand{\n}{\mathbf{n}}

\newcommand{\Oc}{{\mathcal O}}
\newcommand{\oc}{{\mathcal O}}

\newcommand{\ydg}{{}^{\ku G}_{\ku G}\mathcal{YD}}
\newcommand{\ydk}{{}^{\ku K}_{\ku K}\mathcal{YD}}

\newcommand{\Aut}{\operatorname{Aut}}
\newcommand{\Out}{\operatorname{Out}}

\newcommand\card{\operatorname{card}}

\newcommand\sgn{\operatorname{sgn}}

\theoremstyle{plain}

\newtheorem{lema}{Lemma}[section]
\newtheorem{theorem}[lema]{Theorem}

\newtheorem{prop}[lema]{Proposition}
\newtheorem{claim}{Claim}
\newtheorem{question}{Question}
\newtheorem{step}{Step}

\theoremstyle{definition}
\newtheorem{definition}[lema]{Definition}
\newtheorem{exa}[lema]{Example}

\theoremstyle{remark}
\newtheorem{obs}[lema]{Remark}

\newcommand\id{\operatorname{id}}

\newcommand\st{\mathbb S_3}
\newcommand\sk{\mathbb S_4}
\newcommand\sco{\mathbb S_5}

\newcommand\am{\mathbb A_m}

\newcommand\ac{\mathbb A_4}

\newcommand\as{\mathbb A_6}
\newcommand\A{\mathbb A}

\newcommand\sm{\mathbb S_m}

\newcommand\soc{\mathbb S_8}

\newcommand\sei{\mathbb S_6}
\newcommand\sst{\mathbb S_7}

\newcommand\s{\mathbb S}

\def\pf{\begin{proof}}
\def\epf{\end{proof}}

\theoremstyle{remark}

\begin{document}

\renewcommand{\baselinestretch}{1.2}

\thispagestyle{empty}

\title[pointed Hopf algebras over $\am$ and $\sm$]
{Finite-dimensional pointed Hopf algebras with alternating groups
are trivial}

\author[Andruskiewitsch, Fantino, Gra\~na, Vendramin]{N. Andruskiewitsch,
F. Fantino, M. Gra\~na, L. Vendramin}

\thanks{This work was partially supported by ANPCyT-Foncyt, CONICET, Ministerio de Ciencia y
Tecnolog\'{\i}a (C\'ordoba)  and Secyt (UNC)}

\address{\noindent N. A., F. F. : Facultad de Matem\'atica, Astronom\'{\i}a y F\'{\i}sica,
Universidad Nacional de C\'ordoba. CIEM -- CONICET. 
Medina Allende s/n (5000) Ciudad Universitaria, C\'ordoba,
Argentina}
\address{\noindent M. G., L. V. : Departamento de Matem\'atica -- FCEyN,
Universidad de Buenos Aires, Pab. I -- Ciudad Universitaria (1428)
Buenos Aires -- Argentina}
\address{\noindent L. V. : Instituto de Ciencias, Universidad de Gral. Sarmiento, J.M. Gutierrez
1150, Los Polvorines (1653), Buenos Aires -- Argentina  }

\address{}

\email{(andrus, fantino)@famaf.unc.edu.ar} \email{(matiasg,
lvendramin)@dm.uba.ar}

\subjclass[2000]{16W30; 17B37}
\date{\today}

\begin{abstract}
It is  shown that Nichols algebras over alternating groups $\am$
($m\ge 5$) are infinite dimensional.  This proves that any complex
finite dimensional pointed Hopf algebra with group of group-likes
isomorphic to $\am$ is isomorphic to the group algebra.  In a
similar fashion, it is shown that the Nichols algebras over the
symmetric groups $\sm$ are all infinite-dimensional, except maybe
those related to the transpositions considered in \cite{FK}, and
the class of type $(2,3)$ in $\sco$. We also show that any simple
rack $X$ arising from a symmetric group, with the exception of a
small list, collapse, in the sense that the Nichols algebra
$\toba(X, \q)$ is infinite dimensional, $\q$ an arbitrary cocycle.
\end{abstract}
\maketitle

\section{Introduction}\label {0}
\subsection{}
In the early 90's, S. Montgomery raised the question of finding a
\emph{non-trivial} finite-dimensional complex pointed Hopf algebra
$H$ with non-abelian group $G$; here ``non-trivial" means that $H$
is neither a group algebra, nor is  cooked out of a pointed Hopf
algebra with abelian group by some kind of extension. This
question was addressed by A. Milinski and H.-J. Schneider around
1995, who produced two examples, one with $G = \st$, another with
$G = \sk$. The main point was to check that a quadratic algebra
$\toba_m$ built from the conjugacy class of transpositions in
$\sm$ is finite-dimensional. They were able to do it for $m=3,4$
using Gr\"obner bases.  These results were published later in
\cite{MS}. Independently, S. Fomin and K.  N. Kirillov considered
closely related quadratic algebras $\Ee_m$, also constructed from
the transpositions in $\sm$, and they determined the dimensions of
$\Ee_3$, $\Ee_4$ and $\Ee_5$ \cite{FK}.  With the introduction of
the Lifting method, see \cite{AS-cambr}, it became clear that
$\toba_m$ and $\Ee_m$ should be Nichols algebras. This was indeed
checked in \cite{MS} for $m=3,4$ and by the third named author for
$m=5$ \cite{G2}.

\medbreak\subsection{} In this paper, we work over the field $\ku$
of complex numbers. If $G$ denotes a finite group, to classify all
complex pointed Hopf algebras $H$ with group of group-likes
$G(H)\simeq G$ and $\dim H<\infty$, we need to determine the
irreducible Yetter-Drinfeld modules over $\ku G$ such that the
dimension of the corresponding Nichols algebras is finite. In
other words, recalling that irreducible Yetter-Drinfeld modules
are parameterized by pairs $(\Oc, \rho)$ -- $\Oc$ a conjugacy
class of $G$, $\sigma\in \Oc$ fixed, $\rho$ an irreducible
representation of the centralizer $C_{G}(\sigma)$ -- and denoting
by $\toba(\Oc, \rho)$ the associated Nichols algebra, we need to
know for which pairs $(\Oc, \rho)$ is $\dim\toba(\Oc, \rho) <
\infty$.  Assume that $G = \sm$. Then $\toba_m$ and $\Ee_m$
correspond to $\Oc = \Oc_2^m$, the conjugacy class of the
transpositions, and $\rho$ the one-dimensional representations of
the centralizer $\simeq \s_{m-2}\times \s_2$,
$\rho=\epsilon\otimes\sgn$ or $\rho=\sgn\otimes\sgn$ respectively.
If $m\geq 6$, it is still open whether:
\begin{itemize}
    \item $\toba_m$ and $\Ee_m$ are Nichols algebras,
    \item the dimensions of $\toba_m$ and $\Ee_m$ are finite,
    \item the dimensions of $\toba(\Oc_2^m, \sgn^j\otimes \sgn)$, $j=1,2$ are
    finite.
\end{itemize}

\medbreak\subsection{} Recently, there was some progress on
pointed Hopf algebras over $\sm$:
\begin{itemize}
    \item The classification of the finite-dimensional Nichols algebras over
        $\st$, resp. $\sk$, is concluded in \cite{AHS}.
    \item The classification of finite-dimensional pointed Hopf algebras with
        group $\st$, resp.  $\sk$, is concluded in \cite{AHS}, resp. \cite{GG}.
    \item Most of the Nichols algebras $\toba(\Oc, \rho)$ over $\sm$ have infinite
        dimension, with the exception of a short list of open possibilities \cite{AF1, afz}.
\end{itemize}


Our first main result adjusts drastically the list given in \cite[Th. 1]{afz}. See \S\ref{subsect:sm} for the
unexplained notation.

\begin{theorem}\label{th:sm-intro}
    Let $m\ge 5$.  Let $\sigma\in \mathbb S_m$ be of type $(1^{n_1},2^{n_2},\dots,m^{n_m})$,
    let $\Oc$ be the conjugacy class of $\sigma$ and let $\rho=(\rho,V) \in
    \widehat{C_{\mathbb S_m}(\sigma)}$. If $\dim\toba(\Oc, \rho) < \infty$, then
    the type of $\sigma$ and $\rho$ are in the following list:
    \renewcommand{\theenumi}{\roman{enumi}}\renewcommand{\labelenumi}{(\theenumi)}
    \begin{enumerate}
        \item\label{it:caso2}
            $(1^{n_1}, 2)$, $\rho_1 =\sgn$ or $\epsilon$, $\rho_2 =\sgn$.
        \item
            $(2, 3)$ in $\sco$, $\rho_2 =\sgn$, $\rho_3= \overrightarrow{\chi_{0}}$.
        \item\label{it:caso222}
            $(2^3)$ in $\sei$, $\rho_2=\overrightarrow{\chi_{1}}\otimes\epsilon$ or
            $\overrightarrow{\chi_{1}}\otimes \sgn$.
    \end{enumerate}
\end{theorem}
Actually, the rack $\oc_2$ in $\s_6$ is isomorphic to $\oc_{2^3}$,
since any map in the class of the outer automorphism of $\s_6$
applies $(1\;2)$ in $(1\;2)(3\;4)(5\;6)$ \cite{JR}. Thus, case
\eqref{it:caso222} is contained in \eqref{it:caso2}. The remaining
cases can not be treated by consideration of Nichols algebras of
subracks, see Remark \ref{obs:nosubracks}.

\medbreak\subsection{} The main results in this paper are
negative, in the sense that they do not provide any new example of
finite-dimensional pointed Hopf algebra. In fact, very few
examples of finite-dimensional non-trivial pointed Hopf algebras with non-abelian group are known,
see \cite{G2}; at the present moment, it is not clear what is the
class of non-abelian finite groups that may afford
finite-dimensional pointed Hopf algebras. Therefore, it is
important to narrow down as many examples as possible in order to
have a feeling of what this class might be. Here is our second
main result.

\begin{theorem}\label{th:an}
    Let $G = \am$, $m\ge 5$. If $\Oc$ is any conjugacy class of $G$, $\sigma\in \Oc$ is
    fixed and $\rho\in \widehat{C_{G}(\sigma)}$, then $\dim \toba(\Oc, \rho) = \infty$.
\end{theorem}

In order to state the consequences of this result for pointed Hopf
algebras, it is convenient to introduce the following terminology.
\begin{definition}\label{def-intro:gpo-collapses}
    We shall say that a finite group $G$ \emph{collapses} if for any finite-dimensional
    pointed Hopf algebra $H$, with $G(H) \simeq G$, necessarily $H\simeq\ku G$.
\end{definition}

By the Lifting Method \cite{AS-cambr}, we conclude from
Theorem~\ref{th:an}:
\begin{theorem}\label{th:an-pointed}
    If $m\ge 5$, then the alternating group $\am$ collapses. \qed
\end{theorem}

This result was known for the particular cases $m=5$ and $m=7$
\cite{AF2,F-tesis}. We prove it for $m\ge 6$. Since $\A_3$ is
abelian, finite dimensional Nichols algebras over it are
classified, there are $25$ of them; this can be deduced from
\cite[Th. 1.3]{AS-adv}, \cite[Th. 1.8]{AS2}. Nichols algebras over
$\A_4$ are infinite-dimensional except for four pairs
corresponding to the classes of $(1\;2\;3)$ and $(1\;3\;2)$ and
the non-trivial characters of $\Z/3$ \cite[\S 2.2]{AF2}. Actually,
these four algebras are connected to each other either by an outer
automorphism of $\A_4$ or by the Galois group of $\mathbb
Q(\zeta_3)|\mathbb Q$ (the cyclotomic extension by third roots of
unity). Therefore, there is only one pair to study for $\A_4$.

\medbreak\subsection{} Our ultimate goal, towards the
classification of finite-dimensional poin\-ted Hopf algebras, is
to answer the following question.

\begin{question}\label{que:groups}
    For any finite group $G$ and for any $V\in \ydg$, determine if $\dim \toba(V) < \infty$.
\end{question}

Since the category $\ydg$ is semisimple, the question splits into
two cases:
\renewcommand{\theenumi}{\roman{enumi}}   \renewcommand{\labelenumi}{(\theenumi)}
\begin{enumerate}
    \item\label{item:irreducible}
        $V$ irreducible,
    \item \label{item:compl-reducible}
       $V$ completely reducible, \emph{i.~e.} direct sum of (at least 2) irreducibles.
\end{enumerate}

Case \eqref{item:irreducible} was addressed in several recent
papers for some groups and some conjugacy classes \cite{AZ, AF1,
AF2, AF3, afz, FGV, FV}; case \eqref{item:compl-reducible} was
considered in \cite{AHS, HS1}. Of course, the Nichols algebras of
the simple submodules of a completely reducible $V$ such that
$\toba(V)$ is finite-dimensional, should be finite-dimensional
too. But the interaction between the two cases goes also in the
other way. To explain this, we need to recall that Question
\ref{que:groups} can be rephrased in terms of racks. Indeed, the
Nichols algebra of a Yetter-Drinfeld module depends only on its
braiding, which in the case of a group algebra is defined in terms
of the conjugation. A rack is a set with a binary operation
satisfying the basic properties of the conjugation in a group (see
\S\ref{subsect:racks} below).  Then Question \ref{que:groups} is
equivalent to the following one, see \cite{AG1}.

\begin{question}\label{que:racks}
    For any finite rack $X$, for any $n\in \N$, and for any non-principal $2$-cocycle $\q$
    as in page \pageref{eqn:non-ppal-cocycle}, determine if $\dim \toba(X, \q) < \infty$.
\end{question}

In fact, the consideration of Question \ref{que:racks} is more
economical than the consideration of Question \ref{que:groups},
since different Yetter-Drinfeld modules over different groups may
give rise to the same pair $(X, \q)$, $X$ a rack and $\q$ a
2-cocycle. This point of view, advocated in \cite{G1, AG1}, is
analogous to the similar consideration of braided vector spaces of
diagonal type in the classification of finite-dimensional pointed
Hopf algebras with abelian group.

\medbreak The consideration of Question \ref{que:racks} has
another advantage. A basic and useful property of Nichols algebras
says that, if $W$ is a braided subspace of a braided vector space
$V$, then $\toba(W) \hookrightarrow \toba(V)$. For instance,
consider a simple $V = M(\oc, \rho)\in \ydg$ -- say $\dim \rho =
1$ for simplicity. If $X$ is a proper subrack of $\oc$, then
$M(\oc, \rho)$ has a braided subspace of the form $W = (\ku X,
c^q)$, which is clearly not a Yetter-Drinfeld submodule but can be
realized as a Yetter-Drinfeld module over smaller groups, that
could be reducible if $X$ is decomposable. If we know that $\dim
\toba(X, q) = \infty$, say because we have enough information on
one of these smaller groups, then $\dim\toba(\oc,\rho)=\infty$
too.

\medbreak Both Questions have the common drawback that there is no
structure theorem neither for finite groups nor for finite racks.
Therefore, and in order to collect evidence about what groups or
what racks might afford finite-dimensional Nichols algebras, it is
necessary to attack different classes of groups or of racks.
Prominent candidates are the finite simple groups and the finite
simple racks. Finite simple racks have been classified in
\cite[Th. 3.9, Th. 3.12]{AG1} (see also \cite{jo}); explicitly,
any simple rack is isomorphic to one and only one of the
following:

\begin{enumerate}
    \item\label{clasif:permutation-rack} $|X|=p$ a prime, $X\simeq\fp$ a permutation rack, that is $x\trid y=y+1$.
    \item $|X|=p^t$, $p$ a prime, $t\in \N$, $X\simeq({\fp}^t, T)$ is an affine crossed set
        where $T$ is the companion matrix of a monic irreducible polynomial of degree $t$,
        different from $X$ and $X-1$.
    \item $|X|$ is divisible by at least two different primes, and $X$ is twisted
        homogeneous. That is,  there exist a non-abelian simple group $L$, a positive
        integer $t$ and  $x \in \Aut (L^t)$, where $x$ acts by
        $x\cdot(l_1,\ldots,l_t)=(\theta(l_t),l_1,\ldots,l_{t-1})$ for some
        $\theta\in\Aut(L)$, such that $X = {\mathcal O_x}(n)$ is an orbit of the action
        $\rightharpoonup_x$ of $N=L^t$ on itself ($n \neq m^{-1}$ if $t=1$ and $x$ is inner,
        $x(p) = mpm^{-1}$).  Furthermore, $L$ and $t$ are unique, and $x$ only depends on
        its conjugacy class in $\Out (L^t)$.  Here, the action $\rightharpoonup_x$ is given
        by $p\rightharpoonup_xn=p\,n\,(x\cdot p^{-1})$.
\end{enumerate}

In particular, non-trivial conjugacy classes in finite simple
groups, and conjugacy classes in symmetric groups of elements not
in the alternating subgroup are simple racks.
Therefore, it is natural to begin by families of simple groups.

\medbreak\subsection{} To prove Theorems \ref{th:sm-intro} and
\ref{th:an}, we first establish Theorem
~\ref{th:racks-sn-liquidados}, namely that $\dim \toba(X,q) =
\infty$ for many conjugacy classes $X$ in $\sm$ or $\am$ and any
cocycle $q$. This relies on a result on Nichols algebras of
\emph{reducible} Yetter-Drinfeld modules \cite[Th. 8.6]{HS1}, this paper being a sequel to, and based on the results of, \cite{AHS}.
Indeed, let us say (informally) that a rack \emph{collapses} if
$\dim \toba(X,q) = \infty$ for any cocycle $q$; see the precise
statement of this notion in Def. \ref{def:rack-collapses}. To
translate one of the hypothesis of \cite[Th.8.6]{HS1} to
rack-theoretical terms, we introduce the notion of rack of type D.
We deduce from \cite[Th.8.6]{HS1} our Th. \ref{th:racks-claseD},
that says that any rack of type D collapses. It is easy to see
that if $\pi: Z \to X$ is an epimorphism of racks and $X$ is of
type D, then so is $Z$. But any indecomposable rack $Z$ has a
simple quotient $X$; this justifies further why we look at simple
racks, starting with non-trivial conjugacy classes in simple
groups. This is one of the consequences of the study of Nichols
algebras of decomposable Yetter-Drinfeld modules in the analogous
study of indecomposable ones. We stress that the computation of a
second rack-cohomology group is a difficult task. By \cite{EG}, it
coincides with a first group-cohomology group, but this does not
make the problem easier. Two of us have developed a program for
calculations with racks, that in particular computes the
rack-cohomology groups \cite{GV}. The point of view taken in this
article allows to disregard sometimes these considerations.

\medbreak
Actually, we give in Th.~\ref{th:racks-sn-liquidados} a
list of conjugacy class in $\sm$ or $\am$ which are of type D;
hence, if $X$ belongs
to this list and $\pi: Z \to X$ is an epimorphism of racks, then the Nichols algebra
$\toba(Z,q)$ has infinite-dimension for an arbitrary cocycle $q$.

\medbreak To consider the cases left open in
Th.~\ref{th:racks-sn-liquidados} we use
techniques of abelian subracks from our previous papers -- see Lemma~\ref{lema:a4xcr}.

\medbreak The paper is organized as follows. After Section 2 with
Preliminaries, we present our applications of \cite[Th. 8.6]{HS1}
in Section~\ref{sect:hs}. In Section~\ref{section:racks-caen} we
prove Th.~\ref{th:racks-sn-liquidados} and then complete the
proofs of Ths.~\ref{th:sm-intro} and \ref{th:an}.

\medbreak \subsection{Glossary}\label{subsec:glossary} We have
found useful to introduce several notations concerning racks and
groups in relation with the finite-dimensional Hopf algebras and
Nichols algebras. We collect here these new terms.

\begin{itemize}
    \item A finite group $G$ \emph{collapses}\footnote{This was referred to as \emph{of type B} in \cite{FGV}.}
    if for any finite-dimensional
    pointed Hopf algebra $H$, with $G(H) \simeq G$, necessarily $H\simeq\ku
    G$. Equivalently, for any $0\neq V\in \ydg$,    $\dim \toba(V) = \infty$. See Def.
    \ref{def-intro:gpo-collapses}, p.
    \pageref{def-intro:gpo-collapses}.

 \smallbreak\item A finite rack $X$ \emph{collapses} if for any finite faithful cocycle
    $\q$, the Nichols algebra $\toba(X,\q)$ is infinite dimensional. See Def.
    \ref{def:rack-collapses}, p. \pageref{def:rack-collapses}.

 \smallbreak\item A finite rack $X$ is of \emph{type B} if it
 satisfies condition (B) in Remark \ref{obs:collapse}, p.
 \pageref{def:tipob}.

 \smallbreak\item A finite rack $X$ is of \emph{type D} if it
 contains a decomposable subrack $Y = R\coprod S$  such that
$r\trid(s\trid(r\trid s)) \neq s$, for some $r\in R$, $s\in S$. \footnote{Here D stands for decomposable, and B for
the class that produces nothing.}
See Def. \ref{def:tipod}, p. \pageref{def:tipod}.

 \smallbreak\item A finite group $G$ is of \emph{type D} if all its
    non-trivial conjugacy classes are of type D. See \cite{afgv-spor}.
\end{itemize}

\section{Preliminaries}\label{sect:prel}
\subsection{Notation}\label{subsect:not}
Let $G$ be a
group, $\sigma\in G$. We write $\vert G\vert$, resp. $\vert
\sigma\vert$, for the order of $G$, resp., $\sigma$;
$\mathcal{O}_{\sigma}=\mathcal{O}_{\sigma}^{G}$ for the conjugacy
class of $\sigma$ in $G$, with a superscript $G$ if emphasis is
needed. Also, $\widehat{G}$ is the set of isomorphism classes of
irreducible representations of $G$. If $X$ is a set, $\ku X$ is
the vector space with a basis $(e_x)_{x\in X}$.

A \emph{braided vector space} is a pair $(V,c)$, where $V$ is a
vector space and $c\in \GL(V\otimes V)$ is a solution of the braid
equation: $(c\otimes \id) (\id\otimes c) (c\otimes \id) =
(\id\otimes c) (c\otimes \id) (\id\otimes c)$. If $(V,c)$ is a
braided vector space, then $\toba(V)$ denotes its Nichols algebra.
See \cite[p. 22]{AS-cambr}.

\medbreak\subsection{Yetter-Drinfeld modules}\label{subsect:yd}
Let $G$ be a group. A \emph{Yetter-Drinfeld module} over the group
algebra $\ku G$ is a $G$-module $M$ provided with a $G$-grading $M
= \oplus_{g\in G} M_g$ such that $h\cdot M_g = M_{ghg^{-1}}$ for
all $g,h\in G$. The category of Yetter-Drinfeld modules over the
group algebra $\ku G$ is written $\ydg$. This is a braided
category; in particular any $M\in \ydg$ is a braided vector space
with $c\in \GL(M\otimes M)$ given by
\begin{equation}\label{eqn:brqiding-ydg}
    c(m\otimes n) = g\cdot n \otimes m, \qquad \text{for }m\in M_g\;(g\in G),\;n\in M.
\end{equation}

\medbreak The \emph{support} of $M\in \ydg$ is $\supp M = \{g\in
G: M_g \neq 0\}$.

Assume that $G$ is finite. Then the category $\ydg$ is semisimple
and its irreducible objects are parameterized by pairs $(\Oc,
\rho)$, where $\Oc$ a conjugacy class of $G$, $\sigma\in \Oc$
fixed, $\rho$ an irreducible representation of the centralizer
$C_{G}(\sigma)$ of $\sigma$.  If $M(\Oc, \rho)$ denotes the
irreducible Yetter-Drinfeld module corresponding to a pair $(\Oc,
\rho)$ and $V$ is the vector space affording the representation
$\rho$, then $M(\Oc, \rho)$ is the induced module
$\Ind_{C_{G}(\sigma)}^G\rho$ with the grading given by the
identification $\Ind_{C_{G}(\sigma)}^G \rho = \ku
G\otimes_{C_{G}(\sigma)} V \simeq \ku \Oc \otimes_\ku V$. If
$\sigma\in G$ and $\rho\in \widehat{C_{G}(\sigma)}$, then
$\rho(\sigma)$ is a scalar denoted $q_{\sigma\sigma}$. The Nichols
algebra of $M(\Oc, \rho)$ is denoted $\toba(\Oc, \rho)$.

\medbreak Notice that $M(\Oc, \rho)$ can be defined and is a
Yetter-Drinfeld module for any representation $\rho$ of
$C_{G}(\sigma)$.

\medbreak\subsection{Racks}\label{subsect:racks} We briefly recall
the definition and main properties of racks; see \cite{AG1} for
details,  more information and bibliographical references.

\medbreak A \emph{rack} is a pair $(X,\trid)$ where $X$ is a
non-empty set and $\trid:X\times X\to X$ is an operation such that
\begin{eqnarray}
    &\text{the map }\varphi_x = x\trid \underline{\quad}
        \text{ is invertible for any }x\in X,&\ \text{and} \\
    &x \trid(y\trid z) = (x\trid y) \trid(x\trid z)
        \text{ for all }x,y,z\in X.&\label{eqn:selfdist}
\end{eqnarray}

A \emph{morphism} of racks is a map of the underlying sets $f:X\to
Y$ such that $f(x\trid y)=f(x)\trid f(y)$ for all $x,y\in X$. Note
that for any rack $X$, $\varphi:X\to\s_X$ is a morphism of racks.
Really, there is an hierarchy
$$
\{\text{racks}\} \supset \{\text{quandles}\} \supset
\{\text{crossed sets}\},
$$
where a quandle is a rack $X$ such that $x\trid x = x$ for all
$x\in X$; and a crossed set is a quandle $X$ such that $x\trid y =
y$ implies $y\trid x = x$,  for any $x, y\in X$. The permutation rack mentioned
in page \pageref{clasif:permutation-rack}, class (i) of the classification,
is not a quandle. We are only interested in conjugacy classes
and their subracks, which are all crossed sets.

\medbreak Here are some examples and basic notions of racks.

\begin{itemize}
    \item A group $G$ is a rack (actually, a crossed set) with $x\trid y = xyx^{-1}$, $x,y\in G$. Furthermore, if
        $X\subset G$ is stable under conjugation by $G$, that is a union of conjugacy
        classes, then it is a subrack of $G$; e. g., the support of any $M\in \ydg$ is a
        subrack of $G$.

    \smallbreak\item If $A$ is an abelian group and $T\in\Aut(A)$, then $A$ becomes a rack with $x\trid
        y=(1-T)x+Ty$. It will be denoted by $(A,T)$ and called an \emph{affine
        rack}.\label{page:affinerack}

    \smallbreak\item A rack $X$ is \emph{decomposable} iff there exist disjoint subracks
        $X_1,X_2\subset X$ such that $X_i\trid X_j=X_j$ for any $1\le i,j\le 2$ and
        $X=X_1\coprod X_2 $. Otherwise,  $X$ is \emph{indecomposable}.

    \smallbreak    \item A \emph{decomposition} of a rack $X$ is a family $(X_i)_{i\in I}$ of
        pairwise disjoint subracks of $X$ such that $X = \coprod_{i\in I} X_i$ and
        $X\trid X_i = X_i$ for all $i\in I$.

    \smallbreak    \item A rack $X$ is said to be \emph{simple} iff $\card X>1$ and for any surjective
        morphism of racks $\pi: X \to Y$, either $\pi$ is a bijection or $\card Y = 1$.
\end{itemize}

\medbreak\subsection{Cocycles}\label{subsect:rack-cocycles}

Let $X$ be a rack, $n\in \N$. A map $q:X\times X\to\GL(n,\ku)$ is
a \emph{principal 2-cocycle of degree $n$} if
$$q_{x,y\trid z}q_{y,z}= q_{x\trid y,x\trid z}q_{x,z},$$
for all $x,y,z\in X$. Here is an equivalent formulation: let $V=
\ku X\otimes\ku^{n}$ and consider the linear isomorphism
$c^q:V\otimes V\to V\otimes V$,
$$c^q(e_xv\otimes e_yw) = e_{x\trid y}q_{x,y}(w)\otimes e_xv,$$
$x$, $y\in X$, $v$ ,$w\in\ku^{n}$. Then $q$ is a 2-cocycle iff
$c^q$ is a solution of the braid equation. If this is the case,
then the Nichols algebra of $(V, c^q)$ is denoted $\toba(X, q)$.

\medbreak More generally, let $(X_i)_{i\in I}$ be a decomposition
of a rack $X$ and let $\n = (n_i)_{i\in I}$ be a family of natural
numbers. Then a \emph{non-principal 2-cocycle of degree $\n$},
associated to the decomposition $(X_i)_{i\in I}$, is a family $\q
= (q_i)_{i\in I}$ of maps $q_i: X \times X_i \to \GL(n_i,\ku)$
such that
\begin{equation}\label{eqn:non-ppal-cocycle}
    q_i(x,y\trid z)q_i(y,z)= q_i(x\trid y,x\trid z)q_i(x,z),
\end{equation}
for all $x,y\in X$, $z\in X_i$, $i\in I$. Again, this notion is
related to braided vector spaces. Given a family $\q$, let $V=
\oplus_{i\in I} \ku X_i\otimes\ku^{n_i}$ and consider the linear
isomorphism $c^{\q}:V\otimes V\to V\otimes V$,
$$  c^{\q}(e_xv\otimes e_yw)= e_{x\trid y}q_i(x,y)(w)\otimes e_xv,$$
$x\in X_j$, $y\in X_i$, $v\in\ku^{n_j}$, $w\in\ku^{n_i}$. Then
$\q$ is a 2-cocycle iff $c^{\q}$ is a solution of the braid
equation. If this is the case, then the Nichols algebra of $(V,
c^{\q})$ is denoted $\toba(X, \q)$.

Let $X$ be a rack, $\q$ a non-principal 2-cocycle and $V$ as
above. Define a map $g: X \to \GL(V)$ by
\begin{equation}\label{eqn:defgpo}
   g_x(e_yw)= e_{x\trid y}q_i(x,y)(w), \qquad x\in X ,y\in X_i, i\in I.
\end{equation}
Note that $g: X \to \GL(V)$ is a morphism of racks.

\medbreak The next result shows why Nichols algebras associated to
racks and cocycles are important for the classification of pointed
Hopf algebras. It says that Questions \ref{que:groups} and
\ref{que:racks} in the Introduction are indeed equivalent.

\begin{theorem}\label{th:ag414} \cite[Th. 4.14]{AG1}
    (i).  Let $X$ be a finite rack, $(X_i)_{i\in I}$ a decomposition of $X$, $\n \in \N^I$
    and $\q$ a $2$-cocycle as above.  If $G\subset\GL(V)$ is the subgroup generated by
    $(g_x)_{x\in X}$, then $V \in \ydg$.  If the image of $q_i$ generates a finite subgroup
    of $\GL(n_i,\ku)$ for all $i\in I$, then $G$ is finite.

    \medbreak
    (ii). Conversely, if $G$ is a finite group and $V \in\ydg$, then there
    exist a rack $X$, a decomposition $X = \coprod_{i\in I} X_i$, $\n \in \N^I$  and
    non-principal 2-cocycle $\q$ such that $V$ is given as above and the braiding
    $c\in\Aut(V\otimes V)$ in the category $\ydg$ coincides with $c^{\q}$. \qed
\end{theorem}

If $X$ is indecomposable, then there is only one possible
decomposition and only principal 2-cocycles arise. Conversely, the
proof of \cite[Th. 4.14]{AG1} shows that if $V \in\ydg$ as in part
(ii) is irreducible, then the cocycle $\q$ is actually principal.

\medbreak For an easy way of reference, we shall say that a
cocycle $\q$ is \emph{finite} if the image of $q_i$ generates a
finite subgroup of $\GL(n_i,\ku)$ for all $i\in I$.

\medbreak Parallel to the approach to the classification of \fd{}
pointed Hopf algebras \emph{group-by-group}, we envisage the
approach to the classification of \fd{} Nichols algebras
\emph{rack-by-rack}. It is then natural to introduce the following
terminology.

\medbreak Let $X$ be a finite rack and $\q$ a 2-cocycle. First, we
shall say that $(X, \q)$ is \emph{faithful} if the morphism of
racks $g: X \to \GL(V)$ defined in \eqref{eqn:defgpo} is
injective; if $X$ is clear from the context, we shall also say
that $\q$ is faithful. Recall that a rack $X$ is \emph{faithful}
if $\varphi: X \to \s_X$ is injective  \cite[Def. 1.11]{AG1};
clearly, if $X$ is faithful, then $(X, \q)$ is faithful for any
$\q$.

\begin{definition}\label{def:rack-collapses}
    We shall say that a finite rack $X$ \emph{collapses} if for any finite faithful cocycle
    $\q$ (associated to any decomposition of $X$ and of any degree $\n$),
    $\dim\toba(X,\q)=\infty$.
\end{definition}

Here is a useful reformulation of the preceding definition.
\begin{obs}\label{obs:collapse}\label{def:tipob}
    Let $X$ be a finite rack.  Assume that
    \begin{enumerate} \renewcommand{\theenumi}{\alph{enumi}}\renewcommand{\labelenumi}{(\theenumi)}
        \item[(B)]\label{coll-dos}
            For any finite group $G$ and any  $M\in\ydg$ such that $X$ is isomorphic to a
            subrack of $\supp M$, $\dim\toba(M)=\infty$.
    \end{enumerate}
    Then $X$ collapses. The converse is true if $X$ is faithful.
\end{obs}

\pf Assume (B). Let $\q$ be a finite faithful cocycle. By
Th.~\ref{th:ag414} (i), the braided vector space $(V,c^{\q})$
arises from a Yetter-Drinfeld module over a finite group $\Gamma$;
since $\q$ is faithful, $X$ can be identified with $\supp V$.

Now assume that $X$ is faithful and collapses. Let $G$, $M$ as in
(B). The rack $Y$ constructed in Th.~\ref{th:ag414} (ii) is $Y =
\coprod_{i\in I} \oc_i$, where $M = \oplus_{i\in I} M_i$ is a
decomposition in irreducible submodules and $\oc_i = \supp M_i$.
In general, $\supp M \neq Y$, but there is an injective morphism
of racks $\supp M \hookrightarrow Y$, which induces an injective
morphism of racks $X \hookrightarrow Y$. Since $X$ is faithful,
the restriction of the cocycle $\q$ on $Y$ to $X$ is faithful.
\epf

\section{Techniques}\label{sect:hs}
From now on, we shall consider any group $G$ as a rack with the
operation given by conjugation.

\medbreak\subsection{The technique of a suitable
subgroup}\label{sect:subgroup}

If $W$ is a braided subspace of a braided vector space $V$, then
$\toba(W) \hookrightarrow \toba(V)$ \cite[Cor. 2.3]{AS-cambr}. Let
$G$ be a group, $M\in \ydg$. Here are two ways of getting braided
subspaces of $M$:
\begin{itemize}
\item If $Y$ is a subrack of $\supp M$, then $M_Y := \oplus_{y\in Y} M_y$ is a braided subspace of $M$.
\item Let $\sigma\in G$, $H$ a subgroup of $G$ such that $\sigma\in H$.
    If $\rho$ is a representation of $C_{G}(\sigma)$, then
$M(\oc_{\sigma}^{H}, \rho\vert_{C_{H}(\sigma)})$  is a braided
subspace of $M(\oc_{\sigma}^G, \rho)$.
\end{itemize}

\medbreak These ways are actually closed related, by the following
observation.
\begin{lema}\label{lem:general}
    If $Y$ is a subrack of $\supp M$ and $K$ is the subgroup of $G$ generated by $Y$, then
    $M_Y$ is an object in $\ydk$.
\end{lema}
\pf By construction, $M_Y$ is $K$-graded. Furthermore, if $k\in K$
and $y\in Y$, then $k\cdot M_y=M_{k\trid y} \subset M_Y$ since $Y$
is closed under conjugation by $K$. \epf

Assume now that $G$ is a finite group, and let $\sigma$, $\rho$
and $H$ be as above. Then $\rho\vert_{C_{H}(\sigma)} =
\tau_1\oplus\cdots\oplus\tau_s$ where
$\tau_j\in\widehat{C_{H}(\sigma)}$, $1\le j \le s$. Therefore, we
have the following criterium.

\begin{lema}\label{lem:subgrupo_general} Keep the notation above.
    \renewcommand{\theenumi}{\roman{enumi}}   \renewcommand{\labelenumi}{(\theenumi)}
    \begin{enumerate}
        \item\label{item:subgpo1}
            If $\dim \toba(\oc^H,\lambda)=\infty$ for all
            $\lambda\in\widehat{C_{H}(\sigma)}$, then $\dim\toba(\oc^G,\rho)=\infty$ for all
            $\rho\in \widehat{C_{G}(\sigma)}$.

        \medbreak
        \item\label{item:subgpo3}
            Let $\sigma_1,\sigma_2\in\oc^G\cap H$. Let
            $\mathcal{O}_i=\mathcal{O}_{\sigma_i}^H$ and
            assume that $\mathcal{O}_1\ne\mathcal{O}_2$. If
            $\dim\mathfrak{B}(M(\mathcal{O}_{1},\lambda_{1})\oplus
            M(\mathcal{O}_{2},\lambda_{2}))=\infty$ for all
            pairs $\lambda_1\in\widehat{C_H(\sigma_1)}$,
            $\lambda_2\in\widehat{C_H(\sigma_2)}$, then
            $\dim\mathfrak{B}(\mathcal{O}^{G},\rho)=\infty$
            for all $\rho\in \widehat{C_{G}(\sigma)}$. \qed

    \end{enumerate}
\end{lema}

\medbreak\subsection{The splitting
technique}\label{sect:splitconj} We begin by stating the following
result of Heckenberger and Schneider, whose proof uses the main
Theorem of \cite{AHS}.

\begin{theorem}\label{thm:hs} \cite[Th. 8.6 (1)]{HS1}
    Let $G$ be a finite group, $M(\Oc, \rho)$, $M(\Oc', \rho')$ irreducible objects in
    $\ydg$ such that $\dim \toba(M(\Oc, \rho)\oplus M(\Oc', \rho')) <\infty$. Then
    for all $r\in \oc$, $s\in \oc'$, $(rs)^2 = (sr)^2$. \qed
\end{theorem}

We use the previous Theorem as in the following Proposition.
\begin{prop}\label{pr:lamejortec}
    Let $G$ be a finite group and $\mathcal{O}$ a conjugacy class in $G$.  Assume that there
    exist $\sigma_1,\sigma_2$ in $\mathcal{O}$ such that
    $(\sigma_1\sigma_2)^2\ne(\sigma_2\sigma_1)^2$. If there exists a subgroup $H$ such that
    $\sigma_1$ and $\sigma_2$ are not conjugate in $H$, then
    $\mathfrak{B}(\mathcal{O},\rho)=\infty$.
\end{prop}
\begin{proof}
    By Theorem~\ref{thm:hs} and Lemma~\ref{lem:subgrupo_general} (ii).
\end{proof}

We now aim to state a rack-theoretical version of
Prop.~\ref{pr:lamejortec}. Let $G$ be a group, $r,s\in G$. Then
$(rs)^2=(sr)^2 \iff r\trid(s\trid(r\trid s))=s$. We next introduce
a notion that is central in our considerations.

\begin{definition}\label{def:tipod}
    Let $(X, \trid)$ be a rack. We say that $X$ is \emph{of type D} if there exists  a
    decomposable subrack $Y = R\coprod S$ of $X$ such that
    \begin{equation}\label{eqn:hypothesis-subrack}
        r\trid(s\trid(r\trid s)) \neq s, \quad \text{for some } r\in R, s\in S.
    \end{equation}
\end{definition}

\begin{theorem}\label{th:racks-claseD}
    If $X$ is a finite rack of type D, then $X$  collapses.
\end{theorem}
\begin{proof}
    We shall prove that $X$ is, more generally, of type B as in
    Rem. \ref{obs:collapse}.  Let $Y\subseteq X$, $Y=R\coprod S$ a decomposition as in
    Definition~\ref{def:tipod}.  Let $G$ be a finite group, $M\in\ydg$ such that $X$ is
    isomorphic to a subrack of $\supp M$. We identify $X$ to this subrack, and then we can
    take $M_R$ and $M_S$, which are non trivial objects in $\ydk$, $K$ the subgroup of $G$
    generated by $Y$. We may assume that $M_R$ and $M_S$ are irreducible; otherwise, we
    replace them by irreducible submodules.  Now, $\dim\toba(M_R\oplus M_S) = \infty$
    by Th.~\ref{thm:hs}, and then $\dim\toba(M) = \infty$.
\end{proof}

Being of type D is an ubiquitous notion:

    \begin{enumerate}

\medbreak\item If $Y\subseteq X$ is a subrack of type D, then $X$
is of type D.

\medbreak\item If $Z$ is a finite rack and admits a rack
            epimorphism $\pi: Z\to X$, where $X$ is of type D, then $Z$
            is of type D. For, $\pi^{-1}(Y) = \pi^{-1}(R)\coprod \pi^{-1}(S)$ is a decomposable
            subrack of $Z$ satisfying \eqref{eqn:hypothesis-subrack}.
   \end{enumerate}

Let now $X$ be any finite rack. If some indecomposable component
\cite[Prop. 1.17]{AG1} is of type D, then $X$ is of type D. Assume
then that $X$ is indecomposable; then it admits a projection of
racks $\pi: X \to Y$ with $Y$ simple. Thus, it is of primary
interest to solve the following problem.

\begin{question}\label{que:rackssimples-typeD}
Determine all simple racks of type D.
\end{question}

In this paper we consider simple racks arising as conjugacy
classes of the alternating or symmetric groups. In subsequent
papers, we shall investigate other simple racks;  our paper
\cite{afgv-spor} is devoted to conjugacy classes in sporadic
groups.

\medbreak Here are some useful observations to detect conjugacy
classes of type D.
\begin{obs}\label{obs:simplequotients3}
    (a). If $X$ is of type D and $Z$ is a quandle, then $X \times Z$ is of type D.

    (b). Let $K$ be a subgroup of a finite group $G$ and pick $\kappa\in C_G(K)$. We
    consider the map $\R_{\kappa}: K\to G$, $g\mapsto \widetilde g := g\kappa$. Then the
    conjugacy class $\oc$ of $\tau\in K$ can be identified with a subrack of the conjugacy
    class $\widetilde{\oc}$ of $\widetilde{\tau}$ in $G$. Therefore, if $\oc$ is of type D,
    then $\widetilde{\oc}$ is of type D.
\end{obs}

\pf (a). If $r,s\in X$ and $z\in Z$, then $(r,z)\trid (s,z) =
(r\trid s,z)$ because $Z$ is a quandle. The rest is
straightforward. (b) follows because the map $\R_{\kappa}$ is a morphism of racks.
\epf

\begin{exa}\label{exa:jordan}
Let $G$ be a finite group and $g = \tau\kappa\in G$, where $\tau$ and $\kappa\neq e$ commute.
Let $K =  C_G(\kappa) \ni \tau$; then
$\kappa\in C_G(K)$. Hence, the conjugacy class
$\oc$ of $\tau$ in $K$ can be identified with a subrack of the conjugacy
class $\widetilde{\oc}$ of $g$ in $G$ via the
morphism $\R_\kappa$ as in the preceding example.
Assume that
\begin{itemize}
\item $\widetilde{\oc}$ and $\oc$ are quasi-real of type $j$,
\item the orders $N$ of $\tau$ and $M$ of $\kappa$ are coprime,
\item $M$ does not divide $j-1$,
\item there exist $r_0$, $s_0\in \oc$ such that $r_0\trid(s_0\trid(r_0\trid s_0)) \neq s_0$.
\end{itemize}
Then $\widetilde{\oc}$ is of type D.
\end{exa}

\pf Observe first that $(\kappa^j x)\trid (\kappa^h y) = \kappa^h (x\trid y)$
for any $x,y\in K$, $j,h \in \Z$.
Let $R = \R_{\kappa}(\oc)$,
$S = \R_{\kappa^{j}}(\oc)$;  $R$ and $S$ are subracks of $\widetilde\oc$ by Remark \ref{obs:simplequotients3}.
For,
$$
S = \R_{\kappa^{j}}(\oc) \overset{\text{$\oc$ quasi-real}}= \R_{\kappa^{j}}(\oc_{\tau^j}^K) = \oc_{\tau^j\kappa^{j}}^K \subseteq \oc_{\tau^j\kappa^{j}}^G \overset{\text{$\widetilde{\oc}$ quasi-real}}= \widetilde\oc.
$$

We next claim that $R$ and $S$ are disjoint. Indeed, if
$z = \kappa x = \kappa^{j} y$, where $x,y\in \oc$, then the order
of $\kappa^{j-1}  = xy^{-1}$ divides both $N$ and $M$ (note that $x$ and $y$ commute);
hence $\kappa^{j-1} = e$, a contradiction.
Now $r = \kappa r_0\in R$,
$s = \kappa^{j} s_0\in S$ satisfy \eqref{eqn:hypothesis-subrack}.
\epf

\begin{obs}
Theorem \ref{th:racks-claseD} generalizes \cite[Cor. 4.12]{AF3}.
Indeed, if $\oct$ is the octahedral rack and $\oct^{(2)}$ is a
disjoint union of two copies of $\oct$, see \cite{AF3}, then
$\oct^{(2)}$ is of type D. For the other techniques in \cite{AF3},
see Example \ref{exa:dp-tipoD}.
\end{obs}

\bigbreak

\subsection{Some constructions of
racks}\label{subsec:inverserack} We now present a general
construction that might be of independent interest. Let $X$ be a
rack, with operation $x\trid y=\varphi_x(y)$, and let $j$ an
integer. Let $X^{[j]}$ be a disjoint copy of $X$, with a fixed
bijection $X\to X^{[j]}$, $x\mapsto x^{[j]}$, $x\in X$. We define
a multiplication $\trid$ in $X^{[j]}$ by
\begin{equation}\label{eqn:rack^{j}}
x^{[j]}\trid y^{[j]} = (\varphi_x^{j}(y))^{[j]}, \quad x,y\in X.
\end{equation}
Notice that $X^{[j][k]} \simeq X^{[jk]}$, for $j, k\in \Z - 0$.

\begin{lema}\label{lema:rackinverso}
    \begin{enumerate}
        \item $X^{[j]}$ is a rack, called the $j$-th power of $X$.
        \item The disjoint union $X^{[1,j]}$ of $X$ and $X^{[j]}$  with multiplication such
            that $X$ and $X^{[j]}$ are subracks, and
            \begin{equation}\label{eqn:rack^1j}
                x\trid y^{[j]} = (x \trid y)^{[j]}, \quad x^{[j]}\trid y =
                \varphi_x^{j}(y), \quad x,y\in X,
            \end{equation}
            is a rack.
    \end{enumerate}
\end{lema}

$X^{[1,j]}$ is a particular case of an amalgamated sum of racks.
The rack $X^{[-1]}$ will be called the inverse rack of $X$ and
will be denoted $X'$; the corresponding bijection is denoted
$x\mapsto x'$. Note $X'' \simeq X$. The rack $X^{[1,1]}$ will be
denoted $X^{(2)}$ in accordance with \cite{AF3}.

\pf
We first show (a) for $j =-1$. The self-distributivity \eqref{eqn:selfdist} holds iff
$\varphi_x\varphi_y = \varphi_{x\trid y}\varphi_x$ for all $x$, $y\in X$, iff
$\varphi^{-1}_{\varphi^{-1}_x(u)}\varphi^{-1}_x = \varphi^{-1}_x\varphi^{-1}_{u}$ for all
$x$, $u\in X$ (setting $u = x\trid y$); this is in turn equivalent to the
self-distributivity for $X'$. We next show (a) for $j \in \N$. We check recursively that
$\varphi_x\varphi^{j}_y = \varphi^{j}_{x\trid y}\varphi_x$, $\varphi_x^{j}\varphi_y =
\varphi_{\varphi_x^{j}(y)} \varphi_x^{j}$. Hence $\varphi^{j}_x\varphi_y^{j} =
\varphi_{\varphi_x^{j}(y)}^{j}\varphi_x^{j}$, and we have self-distributivity for $X^{[j]}$.
Combining these two cases, we see that self-distributivity holds for $X^{[j]}$, for any
$j\in \Z - 0$. The proof of (b) is straightforward.
\epf

\begin{exa}\label{exa:rack-expandido}
            Let $j\in \Z - 0$. Assume that $X$ is a subrack of $G$ such
            that the map $\eta_j:X \to G$,  $x\mapsto x^{j}$, is inyective.
            Then the image $X^{j}$ of $\eta_j$ is also a subrack, isomorphic
            to the rack $X^{[j]}$. If $X\cap X^{j} = \emptyset$, then the
            disjoint union $X\cup X^{j}$ is a subrack of $G$ isomorphic to
            $X^{[1,j]}$.
\end{exa}

\medbreak\subsection{Affine double racks}\label{subsec:affines}
Let $(A,T)$ be a finite affine rack, see page
\pageref{page:affinerack}. We realize it as a conjugacy class in
the following way. Let $d = \vert T\vert$. Consider the semidirect
product $G=A\rtimes\langle T\rangle$. The conjugation in $G$ gives
\begin{equation}\label{eqn:affinerackmult}
    (v, T^{h})\trid (w, T^j) = (T^{h}(w) + (\id - T^j)(v) , T^j).
\end{equation}

Let $\Q_{A, T}^{j} := \{(w, T^j): w\in A\}$, $j\in \Z/d$, a
subrack of $G$ isomorphic to the affine rack $(A, T^{j})$. Let
$\Q_{A,T}^{[1,j]}$ be the disjoint union
$\Q_{A,T}^1\cup\Q_{A,T}^j$, $j\in \Z/d$; this is a rack with
multiplication \eqref{eqn:affinerackmult}; it is called an
\emph{affine double rack}. If $j\neq 1$, it can be identified with
a subrack of $G$.

\begin{obs}\label{obs:q1jesq1j}
    If $(j)_T=\sum_{i=0}^{j-1}T^i$ is an isomorphism, then $\Q_{A,T}^j\simeq (\Q_{A,T}^1)^{[j]}$.
    Indeed, the map $(\Q_{A,T}^1)^{[j]}\to\Q_{A,T}^j$, $(v,T) \mapsto (v,T)^j=((j)_T v, T^j)$, is a rack
    isomorphism. Hence, $\Q_{A,T}^{[1,j]}$ is isomorphic to $(\mathsf Q_{A, T}^1)^{[1,j]}$,
cf. Lemma \ref{lema:rackinverso}.
\end{obs}

Let $A^T = \ker (\id - T)$ be the subgroup of points fixed by $T$.

\begin{obs}\label{obs:sin-invariantes} Assume that $A^T = 0$. Then $\Q_{A,T} = \Q^1_{A,T}$ is indecomposable
and it does not contain any abelian subrack with more than one
element.
\end{obs}

For, assume that $\Q_{A,T} = R\coprod S$ is a decomposition, with
$(0,T)\in R$. But then $R\ni (v,T)\trid (0,T) = ((\id - T)(v), T)$
for any $v\in A$; since $\id - T$ is bijective, $\Q_{A,T} = R$.
The second claim follows at once from \eqref{eqn:affinerackmult}.

\begin{lema}\label{lema:dobleafincero}
    Let $j\in \Z/d$.
\renewcommand{\theenumi}{\alph{enumi}}\renewcommand{\labelenumi}{(\theenumi)}
The rack $\Q_{A,T}^{[1,j]}$ is of type D,
        provided that
 \begin{equation}\label{eqn:Q1j-colapsa}
        (\id+T^{j+1})(\id-T)\neq 0.
    \end{equation}

\end{lema}

\pf Let $R=\Q_{A,T}^1$, $S=\Q_{A,T}^j$, $r=(0,T)\in R$. Then
$\Q_{A,T}^{[1,j]}=R\coprod S$.
We check
\eqref{eqn:hypothesis-subrack}.  Pick $v\notin \ker
(\id+T^{j+1})(\id-T)$ and $s=(v,T^j)\in S$. Then
\begin{align*}
    r\trid( s \trid ( r\trid s))=((T-T^{j+1}+T^{j+2})(v), T^j)\neq s,
\end{align*}
since $(\id-T+T^{j+1}-T^{j+2})(v)\neq 0$.
\epf

\medbreak\subsection{Applications of affine double
racks}\label{subsec:affines-appl}

If $X$ a rack that contains a subrack isomorphic to
$\Q_{A,T}^{[1,j]}$, for some affine rack satisfying
\eqref{eqn:Q1j-colapsa}, then $X$ is of type D (therefore it
collapses). We now present a way to check this hypothesis.

\begin{definition}
    Let $G$ be a finite group, $\oc$ a conjugacy class in $G$, $\sigma\in \Oc$.
    Classically, $\sigma\in G$ and $\oc$ are \emph{real} if $\sigma^{-1}\in \Oc$.
    If $\sigma$ is conjugated to $\sigma^{j}\neq \sigma$ for some
    $j\in\N$, then we say that $\sigma$ and $\oc$ are \emph{quasi-real of type
    $j$}. Clearly, any real $\sigma$, which is not an involution, is quasi-real of type
    $\vert\sigma\vert - 1$.
\end{definition}

\begin{prop}\label{lema:dobleafin}
    Let $G$ be a finite group, $\oc\subset G$ a conjugacy class which is quasi-real of type
    $j\in\N$. Let $(A,T)$ be an affine rack with $\vert T\vert = d$, and let $\psi:A\to\oc$
    be a monomorphism of racks. If
    $\id-T^j$ is an isomorphism and $\id + T^{j+1} \neq 0$, then $\oc$ is of
    type D.
\end{prop}

\pf
Since $\id-T^j = (\id -T) (j)_T$, both Remarks \ref{obs:q1jesq1j}
and \ref{obs:sin-invariantes} apply. If $Y = \psi(A)$, then $Y\cap
Y^{j} = \emptyset$. If not, pick $y \in Y\cap Y^{j}$, $y = x^j$
for some $x\in Y$. Then $x = x^j$, because $Y$ does not contain
any abelian subrack with more than one element. But this
contradicts the definition of quasi-real. Hence $\oc$ contains a
subrack isomorphic to $\Q_{A,T}^1\cup\Q_{A,T}^j$, which is
isomorphic to $\Q_{A,T}^{[1,j]}$. Now the statement follows from
Lemma~\ref{lema:dobleafincero}. \epf

\emph{Assume for the rest of this Subsection that $(A,T)$ is a
simple affine rack}; that is, $A = \F_{\hspace{-2pt}p}^{\, t}$,
$p$ a prime, and $T\in \GL(t, \F_{\hspace{-2pt}p}) - \{\id\}$ of
order $d$, acting irreducibly.

In this context, Lemma \ref{lema:dobleafincero} specializes as
follows.

\begin{lema}\label{lema:dobleafin-simple}
      If $j\in \Z/d$, then $\Q_{A,T}^{[1,j]}$ is of type D, provided that
    \begin{equation}\label{eqn:s-condition}
j\neq\begin{cases} \frac{d}{2} - 1, &\text{ if $p$ is odd}, \\
d-1 , &\text{ if $p=2$}.
\end{cases}
    \end{equation}
\qed\end{lema}

\begin{exa}\label{exa:dp-tipoD}
Let $G$ be a finite group, $\oc\subset G$ a conjugacy class which
is quasi-real of type $j\in\N$. Let $(A,T)$ be an affine \emph{simple}
rack with $\vert T\vert = d$, and let $\psi:A\to\oc$ be a
monomorphism of racks. If $j\neq\frac{d}{2} - 1$ when $p$ is odd, or
if $j\neq d-1$ when $p=2$, then $\oc$ is of
type D. Notice that the first case in \eqref{eqn:s-condition} always holds if $d$ is odd or 2.
\end{exa}

If $A = \Z/p$, $p$ a prime, and $T$ has order 2, then $\Q_{A, T}^{1}$ is
called a dihedral rack and denoted
$\D_p$ in accordance with \cite[Def. 2.2]{AF3}; thus $\Q_{A,T}^{[1,1]}$
is denoted $\D_p^{(2)}$.
Therefore, the splitting technique includes (without having to resort to look for
cocycles) the case of quasi-real orbits containing  a dihedral
subrack \cite[Cor. 2.9]{AF3}.

\section{Simple racks from $\sm$ and $\am$}\label{section:racks-caen}
\subsection{Notations on symmetric
groups}\label{subsect:sm} Let $\sigma\in \mathbb S_m$. We say that
$\sigma \in \mathbb S_m$ is of type
$(1^{n_1},2^{n_2},\dots,m^{n_m})$ if the decomposition of $\sigma$
as product of disjoint cycles contains $n_j$ cycles of length $j$,
for every $j$, $1\leq j \leq m$. Let $A_j=A_{1,j} \cdots
A_{n_j,j}$ be the product of the $n_j \geq 0$ disjoint $j$-cycles
$A_{1,j}$, \dots, $A_{n_j,j}$ of $\sigma$. Then
\begin{align}\label{eqn:descomp}
    \sigma=A_1 \cdots A_m;
\end{align}
we shall omit $A_j$ when $n_j =0$. The \emph{even} and the
\emph{odd parts} of $\sigma$ are
\begin{align}\label{sigma:even:odd}
    \sigmae:=\prod_{j \text{ even}}A_j,
        \qquad \sigma_{o}:=\prod_{1 < j \text{ odd}}A_j.
\end{align}
Thus, $\sigma = A_1\sigmae \sigma_{o} = \sigmae \sigma_{o}$; we define $\sigma_{o}$ in this
way for simplicity of some statements and proofs. We say also that $\sigma$ has type
$(1^{n_1}, 2^{n_2}, \dots, \so)$, for brevity.


We now recall the notation on representations of the centralizer
needed in the statement of Theorem \ref{th:sm-intro}. See
\cite{afz} for more details. First, the centralizer $\sm^\sigma =
T_1 \times \cdots \times T_m$, where
\begin{align}\label{genofcent}
   T_j =\langle A_{1,j}, \dots ,A_{n_j,j} \rangle \rtimes \langle
   B_{1,j}, \dots , B_{n_j-1,j} \rangle \simeq  (\Z/ j)^{n_j} \rtimes
   \s_{n_j},
\end{align}
$1\leq j \leq m$.  We describe the irreducible representations of
the centralizers. If $\rho=(\rho,V) \in
\widehat{C_{\s_m}(\sigma)}$, then $\rho=\rho_1\otimes \cdots
\otimes \rho_m$, where $\rho_j \in \widehat{T_{j}{\,}}$ has the
form
\begin{equation}\label{formrho2}
   \rho_j=\Ind_{(\Z/{j})^{n_j} \rtimes \mathbb
   S_{n_j}^{\chi_j}}^{(\Z/{j})^{n_j} \rtimes \mathbb S_{n_j}} (\chi_j
   \otimes \mu_j),
\end{equation} with
$\chi_j \in \widehat{(\Z/{j})^{n_j}}$ and $\mu_j \in
\widehat{\s_{n_j}^{\chi_j}}$ -- see \cite[Sect. 8.6]{S}. Here
$\s_{n_j}^{\chi_j}$ is the isotropy subgroup of $\chi_j$ under the
induced action of $\s_{n_j}$ on $\widehat{(\Z/{j})^{n_j}}$.
Actually, $\chi_j$ is of the form
$\chi_{(t_{1,j},\dots,t_{n_j,j})},$ where $0\leq
t_{1,j},\dots,t_{n_j,j}  \leq j-1$ are such that
\begin{align}\label{lost}
   \chi_{(t_{1,j},\dots,t_{n_j,j})}(A_{l,j})=\omega_j^{t_{l,j}},\qquad
   1\leq l \leq n_j,
\end{align}
with $\omega_j : = e ^ {\frac {2\pi i} {j}}$, where $i=
\sqrt{-1}$. Assume that $\deg(\rho)=1$; that is, $\deg(\rho_j)=1$,
for all $j$. Then $\mathbb S_{n_j}^{\chi_j}=\mathbb S_{n_j}$,
$\mu_j=\epsilon$  or $\sgn \in \widehat{\mathbb S_{n_j}}$,  for
all $j$. Hence, we have that $t_{j}:=t_{1,j}=\cdots=t_{n_j,j}$,
for every $j$, and $\rho_j= \chi_j \otimes \mu_j$. In that case,
we will denote $\chi_j=\chi_{(t_j,\dots,t_j)}$ by
$\overrightarrow{\chi_{t_j}}$.

\medbreak\subsection{The collapse of simple racks from $\sm$ or
$\am$}\label{subsection:racks-caen}

Let $m\ge 5$.  In this Subsection, we show that many simple racks
arising as conjugacy classes in $\sm$ or $\am$ collapse. We fix
$\sigma\in \mathbb S_m$ be of type $(1^{n_1}, 2^{n_2}, \dots,
m^{n_m})$ and let
$$\Oc = \begin{cases} (a) \quad \text{the conjugacy class of
$\sigma$ in $\sm$}, &\text{if } \sigma\notin \am,\\ (b) \quad
\text{the conjugacy class of  $\sigma$ in $\am$}, &\text{if }
\sigma\in \am.
\end{cases}$$

Thus $\oc$ is a simple rack. If $\sigma$ in $\am$, then either
$\Oc_\sigma^{\s_m}$ splits as a disjoint union of two orbits in
$\am$, or else $\Oc_\sigma^{\s_m}=\Oc_\sigma^{\A_m}$. This last
possibility arises when either $n_i>0$ for some $i$ even, or else
$n_i>1$ for some $i$ odd. In any case, if $\Oc_\sigma^{\A_m}$ is
of type D, then so is $\Oc_\sigma^{\s_m}$.

\begin{theorem}\label{th:racks-sn-liquidados}
If the type of $\sigma$ is NOT in the list below, then $\oc$ is of
type D, hence it collapses.
\begin{itemize}
        \item[(a)] $(2,3)$;  $(2^3)$;  $(1^n,2)$.
        \item[(b)]
            $(3^2)$; $(2^2,3)$; $(1^n,3)$; $(2^4)$; $(1^2,2^2)$;
            $(1,2^2)$; $(1,p)$, $(p)$ with $p$ prime.
\end{itemize}
\end{theorem}

\pf We proceed in several steps.

\begin{step}\label{step:uno} (Reduction by juxtaposition). Let $m=p+q$, $\mu\in\s_p$,
$\tau\in\s_q$ and $\sigma=\mu\perp\tau\in\sm$ the juxtaposition.
If $\oc_\mu^p$ is of type D, then $\oc_\sigma^m$ also is. In the
same vein, if $\mu\in\A_p$ and its conjugacy class in $\A_p$ is of
type D, and $\tau\in\A_q$, then the conjugacy class of
$\sigma=\mu\perp\tau\in\A_m$ is of type D.
\end{step}

This is so because the inclusion
$\s_p\times\s_q\hookrightarrow\s_m$ induces an inclusion of racks
$\oc_\mu^p\times\oc_\tau^q\hookrightarrow\oc_\sigma^m$. With this
inclusion, the result follows from
Remark~\ref{obs:simplequotients3}.

This statement can be be rewritten as follows: Let $\mu \in
\mathbb S_p$, with $p\le m$ and $\mu$ is of type
$(1^{h_1},2^{h_2},\dots,p^{h_p})$ with $h_j\le n_j$, $1\le j \le
p$, and let $$\Oc' =
\begin{cases} \text{the conjugacy class of $\mu$ in $\s_p$},
&\text{if } \mu\notin \A_p,\\ \text{the conjugacy class of $\mu$
in $\A_p$}, &\text{if } \mu\in \A_p.
\end{cases}.$$ If $\Oc'$ is of type D, then $\Oc$ is of type D.

\begin{step}\label{step:n-even-mayor4}

If the type of $\sigma$ is $(m)$ with $m\ge 6$ not prime,
then $\oc$ is of type D.
\end{step}

Let $\sigma=(1\;2\;3\;4\cdots m)$. We study different cases.

\medbreak \textit{(i). $m\ge 6$ is even.}

If $m=6$, take $\tau=(1\;2\;5\;6\;3\;4)$. It is straightforward to check that
    $\sigma$ and $\tau$ are not conjugate in $H = \langle\sigma,\tau\rangle$ and that
    $(\sigma\tau)^2\neq(\tau\sigma)^2$.

  If $m > 6$, take $\tau=(1\;3)\trid\sigma=(3\;2\;1\;4\cdots m)$. Then
    $\tau\sigma\tau\sigma(1)=1$ and $\sigma\tau\sigma\tau(1)=7$. Therefore, $(\sigma\tau)^2\neq(\tau\sigma)^2$.
    Let $H=\langle\sigma,\tau\rangle$ be the subgroup of $\s_m$ generated by $\sigma$ and
    $\tau$.

Note that $H=\langle\sigma,\tau\rangle=\langle\sigma,\tau\sigma^{-1}\rangle
        =\langle(1\;2\cdots n),(1\;3)(2\;4)\rangle$.
    It is easy to see that the elements in $H$ can be written as products
    $(\mu_1\times\mu_2)\sigma^i$, where $\mu_1\in\s_{\{1,3,5,\dots,m-1\}}$,
    $\mu_2\in\s_{\{2,4,6,\dots,m\}}$ and the signs $\sgn(\mu_1)=\sgn(\mu_2)$.  Now, if
    $x\in\s_m$ is such that $x\trid\sigma=\tau$, then $x=(1\;3)\sigma^i$ for some $i$, which
    does not belong to $H$. Hence, $\sigma$ and $\tau$ belong to different conjugacy classes
    in $H$.

\medbreak \textit{(ii). $m\ge 5$ odd and divisible by a
non-trivial square number.}

Let $m= h^2k$, with $h\ge 3$. Take then $\sigma=(1\;2\cdots m)$.
    For $1\le i\le hk$, let $r_i=(i\;(hk+i)\;(2hk+i)\cdots( (h-1)hk+i))$,
    and consider $\tau=r_1\trid\sigma$. Then $\sigma$ and $\tau$ are conjugate in $\A_m$, but
    they are not conjugate in $H=\langle \sigma,\tau\rangle$. To see this, notice that
    $\sigma r_1^{-1}\sigma^{-1}=r_2^{-1}$, and then, as in the proof of
    case \textit{(i)},
    \[
    H=\langle \sigma,\tau\sigma^{-1}\rangle = \langle \sigma,r_1r_2^{-1}\rangle
    \]
    Then
    $H\subseteq G:=\langle \sigma,r_1,\dots,r_{hk}\rangle$.
    Actually, since $\sigma^{hk}=r_1r_2\cdots r_{hk}$, $G$ is an extension
    \[
    1\to \langle r_1,\dots,r_{hk}\rangle\simeq (\Z/m)^{hk}
        \to G\to \Z/hk\to 1.
    \]
    Any element in $G$ can be written uniquely as a product
    $r_1^{i_1}\cdots r_{hk}^{i_{hk}}\sigma^j$, where $0\le j<hk$. We can consider then the
    homomorphism $\alpha:G\to \Z/h$, given by
    $\alpha(r_1^{i_1}\cdots r_{hk}^{i_{hk}}\sigma^j)=\omega^{i_1+i_2+\cdots+i_{hk}}$,
    where $\omega$ is a generator of $\Z/h$. This homomorphism is well defined, since
    $\alpha(\sigma^{hk})=\alpha(r_1r_2\cdots r_{hk})=\omega^{hk}=1$.  On the other hand, the
    centralizer of $\sigma$ in $\A_m$ is the subgroup generated by $\sigma$.  Thus, for
    $\sigma$ and $\tau$ to be conjugate in $H$, there should exist an integer $j$ such that
    $r_1\sigma^j\in H$. But it is clear that $H$ is in the kernel of $\alpha$, while
    $\alpha(r_1\sigma^j)=\omega$.

\medbreak \textit{(iii). $m\ge 9$ odd and divisible by at least
two primes.}

Express $\sigma = \tau\kappa$, where $\tau\neq e$ and $\kappa\neq
e$ are powers of $\sigma$, the orders $N$ of $\tau$ and $M$ of
$\kappa$ are coprime and $N>3$ is prime. Note that $m=NM$ and the
type of $\kappa$ is $(M^N)$, hence $K =  C_{\A_m}(\kappa) \simeq
(\Z/M)^N \rtimes \A_{N}$. Let $\widetilde{\oc} = \oc_\sigma^{\am}$
and $\oc = \oc_\tau^{K}$. It is known that $\widetilde{\oc}$ is
quasi-real of type $4$.

We will prove that:
\begin{itemize}
\item $\oc$ is quasi-real of type $4$,
\item there exist $r_0$, $s_0\in \oc$ such that $r_0\trid(s_0\trid(r_0\trid s_0)) \neq s_0$.
\end{itemize}

For the first item, we write $\tau =(v, \alpha)\in (\Z/M)^N
\rtimes \A_{N}$. Since $|\tau|=N$, $(\alpha)_N(v)=0$, where
$(\alpha)_N:=\id+\alpha+\alpha^2+\cdots+\alpha^{N-1}$, and
$|\alpha|=N$, i.~e. a $N$-cycle in $\A_N$, because $N$ is prime.
Let $\beta\in \A_N$ such that $\beta\trid \alpha=\alpha^4$. We
will show that there exists $u\in (\Z/M)^N$ such that
$(u,\beta)\trid (v,\alpha)=(v,\alpha)^4$. The last amounts to
$(\alpha)_4v-\beta v=(\id-\alpha^4)u$ for some $u\in(\Z/M)^N$.
Notice that
$\Imm(\id-\alpha^4)=\ker((\alpha^4)_N)=\ker((\alpha)_N)$. 
Thus $(\alpha)_4v-\beta v\in\Imm(\id-\alpha^4)$ if and only if
$(\alpha)_N\beta v=0$. But the last follows from
$(\alpha)_N\beta=\beta(\alpha)_N$.

For the second item, let $r_0 := \tau =(v, \alpha)$. Observe that
$s_0 := (v + (\id-\alpha)(u), \alpha)\in \oc$ for any $u \in
(\Z/M)^N$. Now
\begin{align*}
\big((v, \alpha)(v + (\id - \alpha)(u), \alpha)\big)^2 &= \big((\id + \alpha)(v) + (\alpha-\alpha^2)(u), \alpha^2\big)^2
\\ = \big((\id &+ \alpha^2)(\id + \alpha)(v)  + (\id + \alpha^2)(\alpha-\alpha^2)(u), \alpha^4 \big)
\\
\big((v + (\id - \alpha)(u), \alpha)(v, \alpha)\big)^2 &= ((\id + \alpha)(v) + (\id-\alpha)(u), \alpha^2)^2
\\  = \big((\id &+ \alpha^2)(\id + \alpha)(v)  + (\id + \alpha^2)(\id-\alpha)(u), \alpha^4 \big).
\end{align*}
Then $(r_0s_0)^2 = (s_0r_0)^2$ iff $(\id + \alpha^2)(\id-\alpha)^2(u) = 0$.
Now the order of $\alpha$ is the odd prime $N$, hence there exists some $u$ such
that\footnote{Here one may simply work in a vector space over some quotient field of $\Z/M$.}
$(\id + \alpha^2)(\id-\alpha)^2(u) \neq 0$; thus $(r_0s_0)^2 \neq (s_0r_0)^2$, that is $r_0\trid(s_0\trid(r_0\trid s_0)) \neq s_0$.

Assume that $M>3$.  Now, $\widetilde{\oc}$ is of type $D$ by
Example \ref{exa:jordan}.

Assume that $M=3$. We will see that there exists $j$ such that
$j\equiv 2 \mod 3$ and $\sigma^j\in \widetilde{\oc}$. Let $k$ be
relative prime to $m$ and $\lambda_k:\Z/m\Z\to \Z/m\Z$ the map $i
\mod m\mapsto ki \mod m$. We can think $\lambda_k$ as a
permutation of $\s_m$ in the obvious way. Then
\begin{align}\label{eqn:potsigma}
\lambda_k\trid \sigma=\sigma^k.
\end{align}
Indeed, for all $i$, $1\leq i\leq m$, we have
\begin{align*}
(\lambda_k\trid \sigma)(i)&=\lambda_k \sigma
\lambda_k^{-1}(i)=\lambda_k \sigma(k^{-1}i \mod m)=\lambda_k
(k^{-1}i+1 \mod m)\\&=k(k^{-1}i+1) \mod m=i+k \mod m=\sigma^k(i).
\end{align*}

\begin{claim}\label{claim:equivalencias}
The following are equivalent:\\
\centerline{$(a)$ $\sigma^k\in\widetilde{\oc}$, \qquad $(b)$
$\sgn(\lambda_k)=1$, \qquad $(c)$ $J(k,m)=1$,} where $J(k,m)$
means the Jacobi symbol of $k \mod m$.
\end{claim}
\pf (a) is equivalent to (b) by \eqref{eqn:potsigma} whereas
\cite[Th. 1]{Sz} says that (b) is equivalent to (c). This last
equivalence is Zolotarev's Lemma when $m>2$ is prime -- see
\cite{Z}.\epf

Let $k\in \Z$ such that $k$ is not a quadratic residue modulo $N$;
it is well-known that there are $\frac{N-1}{2}$ of such $k$'s with
$1\leq k\leq N-1$. By Chinese Remainder Theorem there exists $j$,
with $0\leq j< m$, such that $j \equiv 2 \mod 3$ and $j \equiv k
\mod N$.
Thus $J(j,3)=-1=J(j,N)$, and $J(j,m)=J(j,3)J(j,N)=1$. Hence,
$j\equiv 2 \mod 3$ and $\sigma^j\in \widetilde{\oc}$ as desired.

Now, $\tau^j=((\alpha)_jv,\alpha^j)\not\in \oc$ because
$\alpha^j\not\in \oc_{\alpha}^{\A_N}$. So, if we identify $\alpha$
with the $N$-cycle $(1\, 2\, \cdots \, N)$ in $\A_N$ and
$\widetilde{\alpha}:=(1 \, 3)\trid \alpha$, then $\alpha^j$ and
$\widetilde{\alpha}$ are conjugate in $\A_N$ and
$(\alpha\widetilde{\alpha})^2\neq (\widetilde{\alpha}\alpha)^2$.
Set $r_0=\tau$, $s_0=(v,\widetilde{\alpha})$, $R=\oc \cdot
\kappa$, $S=\oc_{\tau^j}^K \cdot \kappa^j=\oc_{\tau^j}^K \cdot
\kappa^{-1}$. Then $s_0\in\oc_{\tau^j}^K$, $(r_0s_0)^2\neq
(s_0r_0)^2$, $R \bigsqcup S$ is an indecomposable subrack of
$\widetilde{\oc}$ and $(rs)^2\neq (sr)^2$, with
$r=r_0\kappa=\sigma$ and $s=s_0\kappa^{-1}$. Therefore,
$\widetilde{\oc}$ is of type $D$.

\begin{step}\label{step:nm-odd}
The class of type $(n,p)$ in $\A_{n+p}$ is of type D if both $n$
and $p$ are  odd, $n\ge 3$ and $p\ge 5$.
\end{step}

We take $\sigma_1=(1\;2\cdots n)(n+1\;n+2\cdots n+p)$,
    $\sigma_2=(1\;2)(n+1\;n+3)\trid\sigma_1$. Then the subgroup
    $H=\langle\sigma_1,\sigma_2\rangle\subseteq\A_n\times\A_p$. Let
    $\pi:\A_n\times\A_p\to\s_p$ be the projection to the second component, and notice that
    $\pi(\sigma_1),\pi(\sigma_2)$ belong to different conjugacy classes in $\A_p$. Then, $(\sigma_1\sigma_2)^2\neq(\sigma_2\sigma_1)^2$ and they
    are not conjugate in $H$, since both statements hold in $\pi(H)$.

\begin{step}\label{step:dos} If the type of $\sigma$ is $(1^2, j)$ with $j > 5$ odd,
then $\oc$ is of type D. \end{step}

The class $\Oc_j^{\s_j}$ splits as a union $\Oc_1 \coprod \Oc_2$
of 2 classes in $\A_j$; if $R = \Oc_1 \times \{(j + 1)(j+2)\}$, $S
= \Oc_2 \times \{(j + 1)(j+2)\}$, then $Y = R\coprod S$ is a
subrack of $\Oc$ and satisfies \eqref{eqn:hypothesis-subrack}
since it generates the subgroup $\A_j$.

\begin{step}\label{step:previo} If the type of $\sigma$ is as in Table
\ref{tab:classes-d-alt-lit}, then $\oc$ is of type D.
\end{step}

This follows from previous work as explained in Table
\ref{tab:classes-d-alt-lit}.
\begin{table}

    \begin{tabular}[hc]{|c|c|c|}
        \hline type & subrack & reference \\ \hline
        $(1, 2, \so)$, $\so\neq \id$ & $\D_3^{(2)}$ & \cite[Ex. 3.9]{AF3} \\
        $(2^3,\so)$, $\so\neq \id$ & $\D_3^{(2)}$ & \cite[Ex. 3.12]{AF3} \\
        $(4, \so)$, $\so\neq \id$ & $\oct^{(2)}$ &  \cite[Prop. 3.7]{afz} \\
        $(4^2)$ & $\oct^{(2)}$ &  \cite[Proof of Prop. 3.5]{afz}\\
        \hline
    \end{tabular}

\

    \caption{Some classes of type D from the literature}
    \label{tab:classes-d-alt-lit}
\end{table}

\begin{step}\label{step:2-j>3} If the type of $\sigma$ is $(2,j)$, where $j>3$ is odd,
then $\oc$ is of type D. \end{step}

Choose $\sigma = (1\, 2)(3\, 4\, 5\, 6\, \cdots\,j+2)$. Set
$\sigma_1=\sigma$, $h=(3\, 5)$ and $\sigma_2= h\trid \sigma$.
Then, $\sigma_2=(1\, 2)(5\, 4\, 3 \, 6 \,\cdots\,j+2)$. We claim
\begin{itemize}
    \item[(a)] $(\sigma_1 \sigma_2)^2 \neq (\sigma_2 \sigma_1)^2$.
    \item[(b)]  $H := \langle\sigma_1,\sigma_2\rangle\simeq\Z/2 \times \A_j$.
    \item[(c)] The conjugacy classes of $\sigma_1$ and $\sigma_2$ in $H$ are
distinct.
\end{itemize}

(a) follows since $(\sigma_2 \sigma_1)^2(3)=3$, whereas $(\sigma_1
\sigma_2)^2(3)=4$ if $j=5$ and $(\sigma_1 \sigma_2)^2(3)=9$ if
$j\geq 7$. (b) follows because the $j$-cycles $(3\, 4\, 5\, 6\,
\dots\,j+2)$ and $(5\, 4\, 3\, 6\, \dots\,j+2)$ generate $\A_I$,
$I=\{3,4,5,\dots,j+2\}$. (c) follows since the conjugacy class of
one cycle of odd length splits into two classes in $\A_I$.

\begin{step}\label{step:2-3^2} If the type of $\sigma$ is $(2,3^2)$,
then $\oc$ is of type D. \end{step}

The set formed by $\sigma_{1}
=(1\,2\,3)(4\,5\,6)(7\,8)$, $\sigma_{2}=(1\,6\,3)(2\,4\,5)(7\,8)$,
$\sigma_{3}=(1\,6\,4)(2\,3\,5)(7\,8)$ and
$\sigma_{4}=(1\,2\,4)(3\,5\,6)(7\,8)$, is the affine rack
associated to the tetrahedron, \emph{i. e.}
$\Q^1_{\F_{\hspace{-2pt}2}^{\, 2},T}$ with $\vert T \vert = 3$.
The Step follows from Ex. \ref{exa:dp-tipoD}.

\begin{step}\label{step:1-4} If the type of $\sigma$ is $(1,4)$, then $\oc$ is of type D. \end{step}

The group $H = \F_5 \rtimes \F_5^{\times}\simeq \F_5 \rtimes \Z/4$
acts on $\F_5$ by translations and dilations; if we identify $\{1,
\dots, 5\}$ with $\F_5$, then the translation by 1 is the 5-cycle
$(1\,2\,3\,4\,5)$,  the dilation by 2 is $\sigma$ and $H$ is
isomorphic to a subgroup of $\sco$. Thus $\oc$ contains a subrack isomorphic to
$\Q^1_{\F_{\hspace{-1pt}5},T}$ with $\vert T \vert = 4$.
The Step follows from Ex. \ref{exa:dp-tipoD}.

\begin{step}\label{step:2-4} If the type of $\sigma$ is $(2,4)$, then $\oc$ is of type D. \end{step}

Let $H$ be the subgroup of $\as$ generated by $(1\,3\,6)$,
$(2\,4\,5)$ and $\sigma$. Since $(1\,3\,6)$ and $(2\,4\,5)$ span a
subgroup isomorphic to $\F_3^2$, $\sigma \trid (1\,3\,6) =
(2\,4\,5)$ and $\sigma \trid (2\,4\,5) = (1\,6\,3)$, we conclude
that $H$ is isomorphic to $\F_3^2\rtimes \langle T\rangle$, where
$T^2 = -\id$. Thus $\oc$ contains a subrack isomorphic to
$\Q^1_{\F_{\hspace{-2pt}3}^{\, 2},T}$ with $\vert T \vert = 4$.
The Step follows from Ex. \ref{exa:dp-tipoD}.

\begin{step}\label{step:small} If the type of $\sigma$ is as in Table \ref{tab:classes-d-sym},
then $\oc$ is of type D.
\end{step}

We list $\sigma_1$, $\sigma_2$ and
$H=\langle\sigma_{1},\sigma_{2}\rangle$ in
Table~\ref{tab:classes-d-sym}; a straightforward computation shows
that $(\sigma_1 \sigma_2)^2 \neq (\sigma_2 \sigma_1)^2$ and that
the conjugacy classes of $\sigma_1$ and $\sigma_2$ in $H$ are
distinct.

\begin{table}
    \centering
    \begin{tabular}[h]{|c|l|l|c|}
        \hline $\Oc$ & $\sigma_1$& $\sigma_2$ & $H=\langle\sigma_{1},\sigma_{2}\rangle$\\ \hline
        $(1^3,2^2)$& (4\;5)(6\;7) & $(1\;2)(3\;7)$ & $\mathbb{D}_{6}$\\
                $(1,3^2)$& $(2\;3\;4)(5\;6\;7)$ & $(1\;2\;5)(3\;4\;6)$ &  $\Z/7 \rtimes \Z/3$ \\
        $(3^3)$& $(1\;2\;3)(4\;5\;6)(7\;8\;9)$ & $(1\;2\;4)(3\;5\;6)(7\;9\;8)$ & $\A_4 \times \Z/3$ \\
        $(2^5)$&  $(1\,2)(3\,4)(5\,6)(7\,8)(9\,10)$ & $(1\,3)(2\,4)(5\,7)(6\,9)(8\,10)$& $\mathbb{D}_{6}$  \\
        $(1,2^3)$ & $(2\;3)(4\;5)(6\;7)$ & $(1\,6)(2\,4)(3\,5)$ &   $\mathbb{D}_{6}$  \\
        \hline
    \end{tabular}

\

    \caption{Some classes of type D in $\am$ or $\s_m$}
    \label{tab:classes-d-sym}
\end{table}

\medbreak\noindent{\bf Final Step.}
 Assume now that $\sigma$ is not of type D; we apply systematically Step \ref{step:uno}. By Step
 \ref{step:n-even-mayor4}, $n_j = 0$ if $j\ge 6$ is not prime.

Assume that $n_4 \neq 0$; then $\sigma$ is of type $(1^{n_1},
2^{n_2}, 4)$ by Step \ref{step:previo}. But $n_1 = 0$ by Step
\ref{step:1-4} and $n_2 = 0$ by Step \ref{step:2-4}; a
contradiction since $m\ge 5$. Hence, $n_4 = 0$.

Assume that $n_3 \neq 0$; then $\sigma$ is of type $(1^{n_1},
2^{n_2},3^{n_3})$ by Step \ref{step:nm-odd}, and either $n_1 = 0$
or else $n_2 = 0$ by Step \ref{step:previo}. If $n_1 = 0$, then
$n_3 = 1$ by Step \ref{step:2-3^2} and $n_2 \leq 2$ by Step
\ref{step:previo}; in other words, only types $(2,3)$ and
$(2^2,3)$ remain. If $n_2 = 0$, then type $(3^2)$ remains, by Step
\ref{step:small}. Also, if $n_1\neq 0$ and $n_2=0$, then type
$(1^{n_1},3)$ remains, by Step \ref{step:small}.

Assume next that $n_3 = 0$; then $\sigma$ is of type $(1^{n_1},
2^{n_2},j^{n_j})$ by Step \ref{step:nm-odd}, with $j \ge 5$ prime
and $n_j = 0$ or 1. Furthermore, $n_2\le 4$ by Step
\ref{step:small}.

If $n_j = 0$ and $n_1 \neq 0$, then $n_2 \le 2$ by Step
\ref{step:small}; but $n_2 = 2$ implies $n_1 \le 2$ by Step
\ref{step:small}. In other words, only types $(2^3)$, $(2^4)$,
$(1, 2^2)$, $(1^2, 2^2)$ and $(1^n,2)$ remain.

If $n_j = 1$, then either $n_1 = 0$ or else $n_2 = 0$ by Step
\ref{step:previo}. The possibility $n_2 \neq 0$ is excluded by
Step \ref{step:2-j>3}. Thus $n_2 = 0$, hence $n_1 \le 1$ by Step
\ref{step:dos}; thus, only types $(j)$ and $(1, j)$, with $j \ge
5$ prime, remain. \epf

\begin{obs}\label{obs:nosubracks}
    The remaining conjugacy classes do not have enough bad subracks to arrive to similar
    conclusions, except $(p)$ and $(1,p)$ where we do not have enough information yet, as said.
    We list the proper subracks generated by two elements in these classes:

    \renewcommand{\theenumi}{\alph{enumi}}\renewcommand{\labelenumi}{(\theenumi)}
    \begin{enumerate}
        \item The proper subracks of the class of type $(1^{m-2},2)$ in $\sm$ are
            all of the form $\Oc_2^{m_1} \coprod \Oc_2^{m_2} \coprod \cdots \coprod
            \Oc_2^{m_s}$; $\Oc_2^{m_i}$ commutes with $\Oc_2^{m_j}$ if $i\neq j$.

        \item The proper subracks generated by two elements
        of the class of type $(1^n,3)$ are either abelian, or isomorphic to the
            racks of the vertices of a tetrahedron, or a cube, or a
            dodecahedron. The same for the class of type $(3^2)$.

        \item The proper subracks of the class of type $(2,3)$ are abelian with either one or two elements.

        \item The proper subracks of the class of type $(1,2^2)$  are
        abelian racks and dihedral racks with $3$ and $5$ elements.

        \item The proper subracks generated by two elements of the class of type  $(1^2,2^2)$
        are abelian racks and dihedral racks with $3$, $4$ and $5$ elements.

          \item The proper subracks generated by two elements of the class of type  $(2^4)$
          are abelian racks and dihedral racks with $3$ and $4$
            elements.

        \item The proper subracks generated by two elements of the class of type $(2^2,3)$  are either abelian racks
        or indecomposable with $20$ elements.
    \end{enumerate}
\end{obs}

We were not able to find out in general whether or not the classes
$(1,p)$, $(p)$ with $p$ prime, are of type D. For instance, the
classes $(5)$, $(7)$, $(11)$ are not of type D, while the classes
$(13)$, $(17)$ and $(31)$ are of type D.

\medbreak Also, the classes $(1,5)$, $(1,11)$ are not of type D,
while the class $(1,7)$ is of type D. More generally, if $p = 2^h
- 1$ is a Mersenne prime, then $(1,p)$ is of type D. For, set
$q=2^h$; then, the group $H = \F_{\hspace{-2pt}q} \rtimes
\F_{\hspace{-2pt}q}^{\times}\simeq \F_{\hspace{-2pt}q} \rtimes \Z/p$ acts on $\F_q$ by
translations and dilations; if we identify $\{1, \dots, q\}$ with
$\F_q$, then $H$ is isomorphic to a subgroup of $\s_q$.

\medbreak\subsection{An abelian subrack with 3
elements}\label{subsection:triangulito} To deal with the remaining
cases, we apply techniques of abelian subracks. We begin by
recording a result that is needed in Lemma~\ref{lema:a4xcr}.
\begin{lema}\label{lem:triangulitos}
Let $G$ be a finite group, $\Oc$ be the conjugacy class of
$\sigma_1$ in $G$ and $(\rho,V) \in \widehat{C_{G}(\sigma_1)}$.
Let $\sigma_2 \neq\sigma_3\in\Oc - \{\sigma_1\}$; let $g_1 = e$,
$g_2$, $g_3\in G$ such that
$\sigma_{i}=g_{i}\sigma_{1}g_{i}^{-1}$, for all $i$. Assume that
\begin{itemize}
\item $\sigma_1^{h}=\sigma_2\sigma_3$ for an odd integer $h$,
\item $g_{3}g_{2}$ and $g_{2}g_{3}$ belong to $C_{G}(\sigma_1)$, and
\item  $\sigma_i\sigma_j=\sigma_j\sigma_i$, $1\leq i,j\leq 3$.
\end{itemize}

Then $\dim\toba(\Oc,\rho)=\infty$, for any
$\rho\in\widehat{C_{G}(\sigma_1)}$.
\end{lema}

\pf Since $\sigma_i\sigma_j=\sigma_j\sigma_i$, there exist $w\in
V-0$ and $\lambda_i\in\ku$ such that
$\rho(\sigma_i)(w)=\lambda_{i}w$ for $i=1,2,3$. For any $1\leq
i,j\leq3$, we call $\gamma_{ij}=g_{j}^{-1}\sigma_i g_{j}$. It is
easy to see that $\gamma_{ij}\in C_{G}(\sigma_1)$ and that
\begin{align*}
\gamma=(\gamma_{ij})=\left(\begin{array}{ccc}
    \sigma_{1} & \sigma_{3} & \sigma_{2}\\
    \sigma_{2} & \sigma_{1} & \sigma_{2}^{h}\sigma_{1}^{-1}\\
    \sigma_{3} & \sigma_{3}^{h}\sigma_{1}^{-1} & \sigma_{1}\end{array}\right).
\end{align*}
Then, $W=\text{span}\{g_{1} w,g_{2} w,g_{3} w\}$ is a braided
vector subspace of $M(\Oc,\rho)$ of abelian type with Dynkin
diagram given by Figure \ref{fi:triangulito}. Assume that
$\dim\toba(\oc,\rho)$ is finite. Then $\lambda_1\neq 1$; also
$\lambda_1^h\neq 1$, for otherwise $g_{2} w,g_{3} w$ span a
braided vector subspace of Cartan type with Dynkin diagram
$A_1^{(1)}$. Thus, we should have $\lambda_{1}=-1$ and $h$ even,
by \cite[Table 2]{H1}, but this is a contradiction to the
hypothesis on $h$. \epf

\begin{figure}[ht]
    \vspace{1cm}
    \begin{align*}
        \setlength{\unitlength}{1.4cm}
        \begin{picture}(1,0)
            \put(0,0){\circle*{.15}} \put(1,1){\circle*{.15}}
            \put(2,0){\circle*{.15}} \put(0,0){\line(1,1){1}}
            \put(0,0){\line(2,0){2}} \put(1,1){\line(1,-1){1}}
            \put(0,.7){$\lambda_1^{h}$}
            \put(1.6,.7){$\lambda_1^{h^2-2}$} \put(0.9,-0.4){$\lambda_1^{h}$}
            \put(-0.4,-0.07){$\lambda_1$} \put(2.2,-0.07){$\lambda_1$}
            \put(0.9,1.2){$\lambda_1$}
        \end{picture}\qquad \qquad
    \end{align*}
\caption{}\label{fi:triangulito}
\end{figure}

\begin{lema}\label{lema:a4xcr}
    \begin{enumerate}
        \item
            Let $r\ge 1$ be odd and let $G=\mathbb{A}_4\times
            \Z/r$, where $\Z/r$ is the cyclic group of order
            $r$, generated by $\tau$.  Let $\Oc$ be the
            conjugacy class of $\sigma= ((1\;2)(3\;4),\tau)$
            in $G$. Then, $\dim\toba{(\Oc,\rho)}=\infty$ for
            every $\rho\in\widehat{C_{G}(\sigma)}$.

        \item\label{prop:2^n2}
            Let $m\ge 5$ and let $\sigma\in\mathbb{A}_{m}$ be
            of type $(1^{n_{1}},2^{n_{2}},\sigmao)$,
            $\mathcal{O}$ the conjugacy class of $\sigma$ in
            $\mathbb{A}_{m}$ and
            $\rho=(\rho,V)\in\widehat{C_{\mathbb{A}_{m}}(\sigma)}$.
            Then $\dim\mathfrak{B}(\mathcal{O},\rho)=\infty$.
    \end{enumerate}
\end{lema}
\begin{proof}
    For the first part, apply Lemma \ref{lem:triangulitos} with $\sigma_1=
    ((1\;2)(3\;4),\tau)$, $\sigma_2=  ((1\;3)(2\;4),\tau)$, $\sigma_3=((1\;4)(2\;3),\tau)$,
    $g_1=e$, $g_2=((1\;3\;2),1)$, $g_3=g_2^{-1}$ and $h=r+2$.

    We prove now the second part. Notice that the result follows from \cite[Th.~2.3]{AF2} if
    $n_2=0$. Otherwise, $n_2=2k$ is even and positive. Let $r$ be the order of $\sigma_o$.
    We claim  $\ac\times \Z/r$  embeds into $\A_m$ in such a way that the class of
    $((1\;2)(3\;4),\tau)\in\A_4\times \Z/r$ is mapped into the class of type
    $(1^{n_{1}},2^{n_{2}},\sigmao)$ in $\A_m$. For this, just take $\sigma$ to be of type
    $\sigma_o$, acting on indices $\{n_1+2n_2+1,\ldots,m\}$, and $\alpha:\Z/r\to\A_m$,
    $\alpha(\tau)=\sigma$.  Let $\delta:\ac\to(\ac)^k$ be the diagonal map, and consider
    $(\ac)^k$ as a subgroup of $\A_m$ acting on indices $\{n_1+1,\ldots,n_1+2n_2\}$. Then,
    $\delta\times\alpha:\ac\times \Z/4\to(\ac)^k\times\A_{m-n_1-2n_2}\subseteq\A_m$, is the
    claimed map.
    \end{proof}

\begin{obs}\label{obs:a4}
    The case $r=1$ of this Lemma is known (see for example \cite[Prop.~2.4]{AF2}) and it is
    used to kill the conjugacy class of involutions in $\mathbb{A}_4$.
\end{obs}

\subsection{Proof of Theorem \ref{th:sm-intro}}

Let $\sigma\in \A_m$; if $\oc_{\sigma}^{ \A_m}$ is of type D,
    then $\oc_{\sigma}^{ \sm}$ is of type D. Then Th. \ref{th:sm-intro}
follows from Th. \ref{th:racks-sn-liquidados} and previous results:
\begin{enumerate}

    \item $(p)$ in $\s_{p}$, $(1,p)$ in $\s_{1+p}$  (with $p$ odd prime),
    $(1^n,3)$ in $\s_{n+3}$, $(3^2)$ in $\s_6$: discarded by \cite[Th. 1]{AZ}.

    \medbreak
    \item $(1,2^2)$ in $\sco$, $(1^2,2^2)$ in $\sei$, $(2^2,3)$ in $\sst$: discarded by  \cite[Th. 1]{AZ}.

    \medbreak
    \item $(2^4)$ in $\soc$: discarded by \cite[Th.~1~(B)~(i)]{AF1}.

    \medbreak
    \item The restrictions on the representations of the remaining classes have been
        explained in \cite{afz}.
\end{enumerate}

\medbreak\subsection{Proof of Theorem \ref{th:an}}

It follows from Theorem \ref{th:racks-sn-liquidados} and the
following considerations:
\begin{enumerate}
    \item $(p)$ in $\A_{p}$, $(1,p)$ in $\A_{1+p}$  (with $p$ odd prime),
    $(1^n,3)$ in $\A_{n+3}$, $(3^2)$ in $\A_6$:
    discarded by \cite[Th. 2.3]{AF2}.

    \medbreak
    \item $(2^2,3)$ in $\A_7$, $(2^4)$ in $\A_8$, $(1^2,2^2)$ in $\A_6$, $(1,2^2)$ in
        $\A_5$: discarded by  Lemma~\ref{lema:a4xcr}~\eqref{prop:2^n2}.
\end{enumerate}

\medbreak\subsection*{Acknowledgements} We have used \cite{GAP} to perform some
computations.  Part of the work of F. F. was done during a visit to the Universidad de
Almer\'ia (supported by Dpto. \'Algebra y An\'alisis Matem\'atico, Univ. de Almer\'ia and the
CONICET); he is grateful to J. Cuadra by his hospitality.  Part of the work of N. A. was
done during a visit to the University of Munich (the travel was supported by the Mathematisches
Institut); he is grateful to H.-J. Schneider and S. Natale by their hospitality.

\end{document}